\documentclass{amsart}
\usepackage{a4wide}

\usepackage[utf8]{inputenc}
\usepackage[T1]{fontenc}
\usepackage{float}

\usepackage{xfrac}

\usepackage{bm}

\usepackage{baskervillef}
\usepackage[baskerville, varg]{newtxmath}
\usepackage[scr=boondox,cal=boondoxo,bb=boondox,frak=euler]{mathalfa}
\usepackage{xcolor}
\definecolor{darkred}{rgb}{0.4,0,0}
\definecolor{darkgreen}{rgb}{0,0.4,0}
\definecolor{darkblue}{rgb}{0,0,0.4}

\definecolor{darkgray}{rgb}{0.55,0.55,0.55}
\definecolor{lightgray}{rgb}{0.9,0.9,0.9}

\usepackage[colorlinks,citecolor=darkgreen,linkcolor=darkred,urlcolor=darkblue,filecolor=darkred]{hyperref}
\usepackage[noabbrev, nameinlink]{cleveref}
\crefname{thm}{theorem}{theorems}
\makeatletter
\def\@cite#1#2{{\m@th\upshape\bfseries%
[{#1\if@tempswa{\m@th\upshape\mdseries, #2}\fi}]}}
\makeatother
\usepackage{subcaption}

\usepackage[style=alphabetic,natbib=true,giveninits=true,url=false,maxbibnames=99]{biblatex}

\DeclareFieldFormat{title}{\mkbibquote{#1}}
\DeclareFieldFormat[misc]{date}{Preprint (#1)}
\AtEveryBibitem{\clearfield{issn}}
\AtEveryCitekey{\clearfield{issn}}

\usepackage{graphicx}
\usepackage{multirow}
\usepackage{enumitem}

\allowdisplaybreaks[4]

\renewcommand{\H}{\mathcal{H}}
\newcommand{\M}{\mathcal{M}}
\renewcommand{\O}{\mathcal{O}}

\newcommand{\CC}{\mathbb{C}}

\newcommand{\NN}{\mathbb{N}}
\newcommand{\QQ}{\mathbb{Q}}
\newcommand{\RR}{\mathbb{R}}
\newcommand{\ZZ}{\mathbb{Z}}

\newcommand{\modelA}{\mathrm{A}}

\newcommand{\modelZ}{\mathrm{Z}}

\renewcommand{\Re}{\mathrm{Re}}
\renewcommand{\Im}{\mathrm{Im}}

\newcommand{\Aut}{\mathrm{Aut}}
\newcommand{\End}{\mathrm{End}}
\newcommand{\Aff}{\text{\normalfont Aff}}
\newcommand{\Prym}{\mathrm{Prym}}

\newcommand{\SL}{\mathrm{SL}}

\newcommand{\GL}{\mathrm{GL}}

\newcommand{\Sym}{\mathrm{Sym}}
\newcommand{\Dih}{\mathrm{Dih}}

\newcommand{\fr}{\mathrm{fr}}

\newcommand{\modn}[1]{\mathbin{\equiv_{#1}}}
\newcommand{\nmodn}[1]{\mathbin{{\not\equiv}_{#1}}}

\newtheorem*{thm*}{Theorem}
\newtheorem{thm}{Theorem}[section]

\newtheorem{lem}[thm]{Lemma}
\crefname{lem}{lemma}{lemmas}
\newtheorem{prop}[thm]{Proposition}
\crefname{prop}{proposition}{propositions}
\newtheorem{cor}[thm]{Corollary}
\theoremstyle{definition}

\theoremstyle{remark}
\newtheorem{rem}[thm]{Remark}

\DeclareFieldFormat{title}{\mkbibquote{#1}}

\makeatletter
\g@addto@macro\bfseries{\boldmath}
\makeatother

\renewcommand{\subset}{\subseteq}

\begin{document}

	\title[Permutation of periodic points on Weierstrass loci]{Permutations of periodic points of Weierstrass Prym eigenforms}
	
	\author[R. Gutiérrez-Romo]{Rodolfo Gutiérrez-Romo}
	\address[Rodolfo Gutiérrez-Romo]{Centro de Modelamiento Matemático, CNRS-IRL 2807, Universidad de Chile, Beauchef 851, Santiago, Chile.}
	\email{g-r@rodol.fo}
	\urladdr{http://rodol.fo}
	
	\author[A. Pardo]{Angel Pardo}
	\address[Angel Pardo]{Departamento de Matemática y Ciencia de la Computación, Universidad de Santiago de Chile, Las Sophoras 173, Estación Central, Santiago, Chile.
}
	\email{angel.pardo@usach.cl (corresponding author)}
	
	
	
	\thanks{This work was supported by the Center for Mathematical Modeling (CMM), ACE210010 and FB210005, BASAL funds for centers of excellence from ANID-Chile, and also by ANID-Chile through the FONDECYT Regular 1221934 grant.}

	\begin{abstract}
        A Weierstrass Prym eigenform is an Abelian differential with a single zero on a Riemann surface possessing some special kinds of symmetries. Such surfaces come equipped with an involution, known as a Prym involution. They were originally discovered by McMullen and only arise in genus $2$, $3$ and $4$. Moreover, they are classified by two invariants: discriminant and spin.
 
        We study how the fixed points for the Prym involution of Weierstrass Prym eigenforms are permuted. In previous work, the authors computed the permutation group induced by affine transformations in the case of genus $2$, showing that they are dihedral groups depending only on the residue class modulo~$8$ of the discriminant $D$.
        In this work, we complete this classification by settling the case of genus $3$, showing that the permutation group induced by the affine group on the set of its three (regular) fixed points is isomorphic to $\Sym_2$ when $D$ is even and a quadratic residue modulo~$16$, and to $\Sym_3$ otherwise. The case of genus $4$ is trivial as the Pyrm involution fixes a single (regular) point.
        In both cases, these same groups arise when considering only parabolic elements of the affine group.

        By recent work of Freedman, when the Teichmüller curve induced by Weierstrass Prym eigenform is not arithmetic, the fixed points of the Prym involution coincide with the periodic points of the surface. Hence, in this case, our result also classifies how periodic points are permuted.
	\end{abstract}
	
	\maketitle
 	\setcounter{tocdepth}{1}

\section{Introduction}
    A \emph{translation surface} is a closed orientable connected surface (without boundary) that is obtained from edge identifications, by translations, of polygons on the plane, up to scissors congruences.
    The layout of a translation surface on the plane induces a natural set of charts on it whose transition maps are translations, known as a \emph{translation atlas}. By means of this atlas, the surface also inherits a (possibly singular) natural flat metric. If they exist, the singularities lie at the vertices of the polygon and possess cone angles that are multiples of $2\pi$. The space of translation surfaces is stratified by the (unordered) list of cone angles at the vertices.
    
    The group $\SL(2, \RR)$ acts naturally on each stratum by deforming the polygon; this action generalizes the action of $\SL(2, \RR)$ on the space $\GL^+(2, \RR)/\SL(2, \ZZ)$ of flat tori. Over the last decades, there has been a great deal of interest about the dynamical, algebraic and geometric properties of this action.
    We refer the reader to the excellent surveys by Zorich~\cite{Zorich:survey} and Forni--Matheus~\cite{Forni-Matheus:survey} for more details.

    If $X$ is a translation surface, its $\SL(2, \RR)$-stabilizer is called the \emph{Veech group} of $X$ and is denoted by $\SL(X)$. Such a group is never cocompact. Nevertheless, $\SL(X)$ may be a lattice (meaning that it may have finite covolume), and in these cases $X$ is said to be a \emph{Veech surface}.

    Assume now that $X$ is a Veech surface and consider the group $\Aff(X)$ of affine transformations on $X$, that is, the group of orientation-preserving diffeomorphisms whose local expressions (with respect to the translation atlas) are affine maps. In this setting, the derivative map $D\colon \Aff(X) \to \SL(X)$ is a finite-to-one group homomorphism, and it is actually a group isomorphism when the flat metric of $X$ has a single singularity.
    
    A point of $X$ with a finite $\Aff(X)$-orbit is said to be \emph{periodic}. Such points arise in important problems about billiard dynamics as, in particular, the finite-blocking and the illumination problems (see, \emph{e.g.}, \cite{Monteil,Lelievre-Monteil-Weiss}). Therefore, periodic points have been extensively studied during recent years (see, \emph{e.g.}, \cite{Apisa,Apisa:marked,Apisa-Wright}).
    Nevertheless, explicitly determining the periodic points of a concrete Veech surface $X$ has proven to be a daunting task.
    Several advances have been made for some classes of Veech surfaces (see, \emph{e.g.}, \cite{Moeller,Apisa-Saavedra-Zhang,Wright:Ward-Veech,Freedman}).
    In this context, a natural question is what (conjugacy class of) permutation group(s) is obtained from the $\Aff(X)$-action on the set of periodic points (see, e.g., the second question at the end of \cite[Section~1]{McMullen:regular_polygons}).
    This work and our previous work \cite{Gutierrez-Romo--Pardo:H2} answers this question for a particular class of Veech surfaces.

    In genus two, there are exactly two strata of translation surfaces: $\H(2)$ and $\H(1, 1)$. As of yet, a full classification of Veech surfaces in the latter stratum is unknown. On the other hand, Veech surfaces in $\H(2)$, together with their $\SL(2,\RR)$-orbits, were fully classified by McMullen (see also \cite{Calta,Hubert-Lelievre}). The classification relies on two numerical invariants: a real quadratic \emph{discriminant} $D \geq 5$, and a \emph{spin} invariant (taking values $\pm 1$) when $D \modn{8} 1$. The set of periodic points on such a surface $X$ will depend on whether $X$ is in the $\SL(2,\RR)$-orbit of a \emph{square-tiled surface} (that is, a finite cover of the unit flat torus, branched over a single point). Indeed, Möller~\cite{Moeller} showed that,
    this set coincides with the six Weierstrass points (the fixed points of the hyperelliptic involution) in the underlying Riemann surface when $X$ is not in the $\SL(2, \RR)$-orbit of a square-tiled surface. In $\H(2)$, these six points include the singularity (of cone angle $6\pi$), so there exist five regular periodic points. On the other hand, if $X$ belongs to the $\SL(2, \RR)$-orbit of a square-tiled surface, the set of periodic points is countable since it also includes all points that project to torsion points on the unit flat torus. With this in mind, the authors computed the conjugacy classes of permutation groups acting on the five regular Weierstrass points in $\H(2)$:
    
    \begin{thm}[{\cite[Theorem~1.1]{Gutierrez-Romo--Pardo:H2}}]
        Let $X$ be a Veech surface in $\H(2)$ and let $D$ be the real quadratic discriminant arising in McMullen's classification of its $\SL(2, \RR)$-orbit. Consider the permutation group $G(X)$ obtained by the action of the affine group $\Aff(X)$ on the five regular Weierstrass points $\{w_1, \dotsc, w_5\}$ of $X$. We have that the conjugacy class of $G(X)$ inside $\Sym(\{w_1, \dots, w_5\})$ is equal to the conjugacy class of
        \begin{itemize}
            \item $\Dih(\{w_1, \dotsc, w_4\})$ if $D \modn{4} 0$;
            \item $\Dih(\{w_1, \dotsc, w_5\})$ if $D \modn{8} 5$; and
            \item $\Sym(\{w_1, w_2\}) \times \Sym(\{w_3, w_4, w_5\})$ if $D \modn{8} 1$, regardless of the spin;
        \end{itemize}
        where $\Sym(A)$ is the symmetric group acting on a set $A$, and $\Dih(A)$ is the dihedral group acting on a set of vertices $A$. These groups are isomorphic to $\Dih_4$, $\Dih_5$ and $\Dih_6$, respectively, where $\Dih_n$ denotes the dihedral group of order $2n$.
        \label{thm:main_H(2)}
    \end{thm}

    The classification of Veech surfaces in $\H(2)$ also inspires a construction of Veech surfaces of genus three or four possessing a single singularity, known as \emph{Weierstrass Prym eigenforms}. We describe this construction in \Cref{sec:background}. Such surfaces come equipped with an involution, called the \emph{Prym involution}. The Prym involution has exactly $10-2g$ fixed points, where $g$ denotes the genus, and one of them has to be the singularity. In particular, Weierstrass Prym eigenforms only exist in these three genera (that is, two, three and four).
    
    Weierstrass Prym eigenforms were first considered by McMullen~\cite{McMullen:Prym}, and then fully classified by Lanneau--Nguyen~\cite{Lanneau-Nguyen:H4,Lanneau-Nguyen:H6}.
    As in the case of genus two, they are also distinguished by two numerical invariants: a real quadratic \emph{discriminant} $D$ and, in some cases, an analogue of the spin invariant.

    The fixed points of the Prym involution are always periodic points. As such, they play a similar role as Weierstrass points in genus two. Indeed, in this latter case the Prym involution actually coincides with the hyperelliptic involution.
    Recently, Freedman~\cite{Freedman} classified the periodic points on non-arithmetic Weierstrass Prym eigenforms, showing that there are no more such points than previously expected. More precisely, he shows a result analogous to the one by Möller~\cite{Moeller} in genus two: the set of periodic points of a 
    Veech surface is precisely the set of fixed points of the Prym involution when the surface does not belong to the $\SL(2, \RR)$-orbit of a square-tiled surface. Otherwise, the points projecting to torsion points of the unit flat torus are also periodic points.
    
    Given a Weierstrass Prym eigenform $X$, we are interested in the permutation group $G(X)$ induced by the action of $\Aff(X)$ on the regular fixed points of the Prym involution. As previously stated, the case of genus two was settled by the authors~\cite{Gutierrez-Romo--Pardo:H2}, obtaining \Cref{thm:main_H(2)}. The case of genus four is trivial, since there is a single regular fixed point. Thus, it only remains to address the case of genus three. In this work, we settle this last case to classify the conjugacy class of $G(X)$:
    
    \begin{thm}\label{thm:main_H(4)}
        Let $X$ be a Weierstrass Prym eigenform of genus three and let $D$ be the real quadratic discriminant arising in Lanneau--Nguyen's classification of its $\SL(2, \RR)$-orbit. Consider the permutation group $G(X)$ obtained by the action of the affine group $\Aff(X)$ on the three regular fixed points $\{w_1, w_2, w_3\}$ of the Prym involution of $X$. We have that the conjugacy class of $G(X)$ inside $\Sym(\{w_1, w_2, w_3\})$ is equal to the conjugacy class of
        \begin{itemize}
            \item $\Sym(\{w_1, w_2\})$ if $D$ is even and a quadratic residue modulo $16$;
            \item $\Sym(\{w_1, w_2, w_3\})$ otherwise;
        \end{itemize}
        where $\Sym(A)$ is the symmetric group acting on a set $A$.
    \end{thm}

    As in the case of genus two, the conjugacy class of $G(X)$ is in some sense a ``weak'' invariant of the $\SL(2,\RR)$-orbit of $X$: it only depends on the residue class (modulo $16$) of the discriminant.
    Furthermore, the classification does not depend on whether the Veech surface is arithmetic or not, although the proofs use slightly different methods for these two cases.
    
    \subsubsection*{The HLK-invariant}
        A Weiestrass Prym eigenform $X$ in the $\SL(2,\RR)$-orbit of a square-tiled surface is \emph{arithmetic}, that is, it corresponds to a branched cover of a flat torus (elliptic curve), branched over the $1$-torsion point.
        It follows that any regular point fixed by an involution on $X$ projects down to a $2$-torsion point on the torus.
        Moreover, any element of $\SL(2, \RR)$ preserves both the sets of $1$-torsion points and of $2$-torsion points.
        These observations give rise to an invariant for the $\SL(2, \RR)$-orbit of $X$, defined as the combinatorics of the projections of the Prym fixed points on the covered torus.
        In the hyperelliptic case, this invariant is known as the \emph{HLK-invariant}; it was first considered by Kani~\cite{Kani} and Hubert--Lelièvre~\cite{Hubert-Lelievre} (see also the work by Matheus--Möller--Yoccoz \cite{Matheus-Moeller-Yoccoz}).
        In our context, the HLK-invariant naturally restricts the action of $\Aff(X)$ on the set of regular Prym fixed points and, \Cref{thm:main_H(2),thm:main_H(4)} show that, in fact, this is the only restriction.

        In the non-arithmetic case, a suitable reduction allows us to apply the same criterion in some cases to constrain the action of the affine group on regular Prym fixed points, even if the periods are no longer integral; they belong to a real quadratic field. We show that these restrictions are, in fact, the only ones.
        
        \subsubsection*{Generators}
        In our proof, we only make use of affine Dehn multitwists, which are in correspondence with parabolic elements in the Veech group.
        As a consequence, we answer the natural follow-up question on which elements of the affine group are required to generate the group $G(X)$ in \Cref{thm:main_H(4)}.
        \begin{cor}\label{cor:main_H(4)}
        	The group $G(X)$ can be generated by the action of affine multitwists in $\Aff(X)$.
        \end{cor}

    \subsection*{Additional results}
        We include some other results about Weierstrass Prym eigenforms in genus three, which address two important questions that have not previously been answered in the literature. However, the methods we use to deal with these questions are somewhat unrelated to those used to prove our main results. Thus, these additional results are presented as appendices and described below.
        
        In \Cref{app:cylinders}, we extend the near-complete classification of one and two-cylinder decompositions for Weierstrass Prym eigenforms in genus three by Lanneau--Nguyen \cite[Appendix~4]{Lanneau-Nguyen:H4}.
        We give an elementary proof to the fact that these decompositions can only arise in the arithmetic case, and use the HLK-invariant to show that surfaces having one and two-cylinder decompositions lie in different $\SL(2,\RR)$-orbits when $D \modn{8} 1$.
        
        In \Cref{app:representation}, we show that the representation $\SL(X) \to \Aut(H_1^{(0)})$ is never faithful for Weierstrass Prym eigenforms in genus three, where $H_1^{(0)} \subseteq H_1(X; \ZZ)$ is the \emph{holonomy-free submodule}, that is, the kernel of the holonomy map $\mathrm{hol} \colon H_1(X; \ZZ) \to \CC$ (which is invariant under the action of $\SL(X)$). The representation $\SL(X) \to \Aut(H_1(X; \ZZ))$ is referred to as the \emph{Kontsevich--Zorich representation}.

    \subsection{Strategy of the proof}
    We follow a similar strategy of that in the genus~$2$ case treated in \cite{Gutierrez-Romo--Pardo:H2}.
    Lanneau--Nguyen's classification~\cite{Lanneau-Nguyen:H4} of Weierstrass Prym eigenforms relies on a combinatorial description of cylinder cusps, called \emph{splitting prototypes}, following the pioneering ideas of McMullen's classification of Teichmüller curves in $\H(2)$.
    This classification gives rise to \emph{prototypical surfaces} on each orbit of a Veech surface. Since, for any Veech surface $X$, the groups $G(X)$ and $G(g \cdot X)$ are conjugate for every $g \in \SL(2, \RR)$, it is enough to compute $G(X)$ for prototypical surfaces.
        
    Given a prototypical surface $S$, we start by computing the action of some parabolic elements of $\SL(S)$ on its regular Prym fixed points. This shows that $G(S)$ contains at least one transposition. The case of discriminant $D=8$ has to be treated separately, as the prototypical surfaces used for the general argument do not arise in this particular case.

    When $D$ is odd, or an even quadratic nonresidue modulo~$16$, this arguments already shows that $G(S)$ equals its entire ambient symmetric group $\Sym_3$, since the action of the horizontal and vertical affine multitwists twists on $S$ produce two different transpositions.
    The cases of discriminant $D\in\{17,25\}$ are also treated separately due to other minor technical reasons.

    When $D$ is an even quadratic residue modulo~$16$, the previous is no longer true and we use the geometric restrictions imposed by the locations of the Prym fixed points, and the algebraic restrictions imposed by the coefficients of $\SL(S)$, to show that $G(S)$ contains only a transposition and the identity.
    In the arithmetic case, these restrictions are those given by the HLK-invariant. In the non-arithmetic case, we are able to even extend this invariant and obtain the desired restriction in this way.
    This concludes the proof of \Cref{thm:main_H(4)} and \Cref{cor:main_H(4)}.
    
    \subsection{Structure of the paper}
    In \Cref{sec:background}, we introduce the general background necessary to formulate and prove our results. 
    In \Cref{sec:prototypes}, we introduce splitting prototypes and prototypical surfaces. \Cref{sec:parabolic} is devoted to the study of the action of affine multitwists on regular Prym fixed points, computing their action on prototypical surfaces.
    In \Cref{sec:restrictions}, we show that in the case of even discriminants which are quadratic residue modulo~$16$, there is only one non-trivial transposition; this section is divided into several subsections. Indeed, in \Cref{sec:arithmetic} we study the arithmetic case and compute the distribution of regular Prym fixed points over $2$-torsion points in the unit torus, giving the corresponding restrictions for the action on regular Prym fixed points.
    In \Cref{sec:non-arithmetic}, we study the non-arithmetic case, extending these ideas even if there is no such a projection over the torus.
    
    For some small discriminants the general arguments are insufficient; we treat these exceptional cases separately in \Cref{sec:remaining}.
    Finally, in \Cref{sec:proof}, we summarise our results and give a proof of \Cref{thm:main_H(4)} and \Cref{cor:main_H(4)}.

    We also include two appendices. In \Cref{app:cylinders}, we complete the classification of cylinder decompositions for arithmetic Weierstrass Prym eigenforms in genus three.
    In \Cref{app:representation}, we show the failure of the Kontsevich--Zorich representation to be faithful for Weierstrass Prym eigenforms in genus three.

\section{Background}
\label{sec:background}
\subsection*{Translation surfaces}
    A \emph{translation surface} is an orientable closed connected surface that can be obtained by edge-to-edge gluing of polygons in $\CC$ using translations only.
    There is a one-to-one correspondence between translation surfaces and nonzero holomorphic $1$-forms on Riemann surfaces. For more details, we refer the reader to the surfaces by Zorich~\cite{Zorich:survey} and Forni--Matheus~\cite{Forni-Matheus:survey}.

    A translation surface naturally carries a flat metric inherited from $\CC$, with conical singularities with cone angles that are multiples of $2\pi$, and located at the vertices of the polygons. Equivalently, a conical singularity with cone angle $2\pi(k+1)$ corresponds to a zero of order $k$ of the Abelian differential.

    These surfaces can be packaged together into a moduli space. The topology is given by deforming sides of the polygons, or equivalently by period coordinates \cite[Section~3]{Zorich:survey}. This space is stratified by the unordered list of orders of zeros: the surfaces in the stratum \emph{stratum} $\H(\kappa_1, \dotsc, \kappa_n)$ are those with prescribed orders $\kappa_1, \dotsc, \kappa_n$ (counting multiplicities). We can recover the genus from the combinatorics by the relation $\kappa_1 + \dotsb + \kappa_n = 2g - 2$.

    The group $\SL(2, \RR)$ (and, more generally, $\GL^+(2, \RR)$) acts naturally on $\CC$ by the identification $\CC \cong \RR^2$, and this action is carried to the moduli space of translation surfaces. Equivalently, it acts by post-composition with coordinate charts of the Riemann surface. This action preserves both the area of a surface and the stratum in which it lies.

    Let $X$ be a translation surface and let $\SL(X)$ be the stabiliser of the $\SL(2, \RR)$-action on $X$. For a ``generic'' surface, the group $\SL(X)$ is either trivial or equal to $\pm \mathrm{Id}$; the latter case arises if the stratum consists only of hyperelliptic surfaces. Nevertheless, for some very special surfaces the group $\SL(X)$ is ``large'': a lattice in $\SL(2, \RR)$ in the sense that it has finite covolume. Such surfaces are known as \emph{lattice surfaces} or \emph{Veech surfaces}. Moreover, in unpublished work, Smillie showed that Veech surfaces are exactly those with closed $\SL(2, \RR)$-orbits.    
    
    By forgetting the differentials, we can project the $\SL(2, \RR)$-orbit of a Veech surface into the moduli space of Riemann surfaces. The resulting space is a (complex) Teichmüller geodesic and is hence known as a \emph{Teichmüller curve}. There are two kinds of such curves:
    \begin{itemize}
        \item the \emph{arithmetic} ones, commensurable to the modular curve, and arising from orbits of square-tiled surfaces (those that are covers of the unit torus, branched over a single point); and
        \item the \emph{non-arithmetic} ones, which do not arise from square-tiled surfaces.
    \end{itemize}

    All these properties of a surface remain the same when a matrix in $\GL^+(2, \RR)$ is applied. Indeed, the only difference between $\SL(2, \RR)$ and $\GL^+(2, \RR)$ is the inclusion of homotheties, which do not change whether a surface is Veech, arithmetic, etc. 
    
\subsection*{Affine transformations}
    An \emph{affine transformation} on a translation surface is an orientation-preserving self-diffeomorphism whose local expressions (on the natural ``flat'' charts) are affine maps of $\RR^2 \cong \CC$. These maps form a group denoted by $\Aff(X)$.

    Since $f \in \Aff(X)$ is affine, it can be described in any flat chart by the composition of a matrix in $Df \in \SL(2, \RR)$ and a translation. Since all edge identification on a translation surfaces are by translations, the matrix $Df$ does not depend on the flat chart. Hence, we obtain a group homomorphism $D \colon \Aff(X) \to \SL(2,\RR)$, known as the \emph{derivative map}. We have that $D \Aff(X) = \SL(X)$, where we recall that $\SL(X)$ is the Veech group of $X$. The kernel of this homomorphism is finite, and it is known as the \emph{automorphism group} of $X$, that is, the self-diffeomorphisms of $X$ whose local expressions are translations. In this way, we obtain a short exact sequence
    \[
        1 \to \Aut(X) \to \Aff(X) \to \SL(X) \to 1.
    \]
    
    We will shortly restrict our analysis to particular cases of Veech surfaces, all possessing a single zero (that is, belonging to a \emph{minimal stratum}). It is well-known that the automorphism group of any Veech surface belonging to a minimal stratum is trivial (see, for example, \cite[Proposition~4.4]{Hubert-Lelievre}), so the short exact sequence above gives rise to a natural isomorphism $\Aff(X) \cong \SL(X)$. We will tacitly use this identification in the remainder of the article.

\subsection*{Periodic points} The $\Aff(X)$-orbit of a point $x \in X$ may be finite or infinite. In the former case, we say that $x$ is a \emph{periodic point} of $X$; these points are arise in many important questions about billiard dynamics, such as the illumination or finite-blocking problems (see, for example, \cite{Monteil, Lelievre-Monteil-Weiss}). The goal of this work is to classify the $\Aff(X)$-orbits of periodic points for a class of Veech surfaces.
	
	\subsection*{Prym varieties and eigenforms}
	\label{sec:preliminaires_H(4)}
    Fix a Riemann surface $X$. Recall that $H^{1,0}(X)$ is the complex vector space of Abelian differentials on $X$. As a real vector space, $H^{1,0}(X)$ is canonically isomorphic to $H^1(X; \RR)$ via the isomorphism $\omega \mapsto \Re(\omega)$. This isomorphism also endows $H^{1,0}(X)$ with a symplectic form by pulling back the intersection form on $H^1(X; \RR)$.
    
    Suppose that $\iota\colon X \to X$ is a holomorphic involution with $\iota^* \omega = -\omega$. Then, the \emph{Prym variety} of $(X,\iota)$ is the subtorus of the Jacobian of $X$ determined by the anti-invariant eigenspace of $\iota$:
    \[
        \Prym(X,\iota) = (H^{1,0}(X)^-)^*/(H_1(X; \ZZ)^-).
    \]
    
    An Abelian variety $A$ admits \emph{real multiplication} by $\O_D$ (the real quadratic order of discriminant $D > 0$) if $\dim_\CC A = 2$ and $\O_D$ occurs as an indivisible, self-adjoint subring of $\End(A)$.
    
    Now suppose that $P = \Prym(X,\iota)$ is a Prym variety with real multiplication by $\O_D$.
    Then $\O_D$ also acts on $H^{1,0}(X)^-$ and we say that $\omega \in H^{1,0}(X)^-$ is a \emph{Prym eigenform} (for real multiplication by $\O_D$) if $\O_D \cdot \omega \subseteq \CC \cdot \omega$ (and $\omega \neq 0$).
    We say that $\iota$ is the \emph{Prym involution}.
    
    Let $E_D^g \subset \Omega\M_g$ denote the space of all genus-$g$ Prym eigenforms for real multiplication by $\O_D$.
    By the seminal work of McMullen~\cite{McMullen:Prym}, $E_D^g$ is a closed, $\GL^+(2,\RR)$-invariant subset of $\Omega\M_g$.
    
    Note that $2 = \dim_\CC \Prym(X,\iota) \leq g$ and, by Riemann--Hurwitz formula, $\iota$ has exactly $10-2g$ fixed points. Thus, we have that $2 \leq g \leq 5$.
    
    In this work, we focus our attention on \emph{Weierstrass Prym eigenform loci}, that is, loci of Prym eigenforms with a single zero. We define $E_D(2g-2) = E_D^g \cap \H(2g-2)$.
    By a dimension count, $E_D(2g-2)$ consists entirely of Veech surfaces (see McMullen's work \cite[Theorem~1.2]{McMullen:Prym} and \Cref{thm:Prym-classification} below).

    The set of regular points fixed by the Prym involution for a Weierstrass Prym eigenform is fixed by the $\Aff(X)$-action and, thus, they are in particular periodic points.
    In the non-arithmetic case, they are, in fact, the only ones, as showed by Möller~\cite{Moeller} in the genus~$2$ case and, by Freedman~\cite{Freedman}, in genus~$3$ and $4$.
    
    We are interested in the action of $\Aff(X)$ on the set of regular fixed points for the Prym involution.
    
    Weierstrass Prym eigenforms in genus two correspond exactly to the Veech surfaces in $\H(2)$, the Prym involution corresponds to the hyperelliptic involution and its fixed points are the Weierstrass points of the underlying Riemann surface.
    In general, for a Prym eigenform in $E_D(2g-2)$, the only zero of order $2g-2$ is necessarily fixed by the corresponding Prym involution.
    Thus, in higher genera, we have the following:
    \begin{itemize}
        \item in genus three, Prym eigenforms in $E_D(4)$ have exactly three fixed regular points for the Prym involution;
        \item in genus four, Prym eigenforms in $E_D(6)$ have only one fixed regular point for the Prym involution; and
        \item in genus five, $E_D(8) = \emptyset$.
    \end{itemize}
    We conclude that the action of the Veech group on the regular fixed points of the Prym involution can be non-trivial only in genus two and three. The genus-two case was already established by the authors in \cite{Gutierrez-Romo--Pardo:H2}. Thus, we focus our attention to Weierstrass Prym eigenform loci in genus three. To this end, we will make use of the classification of Weierstrass Prym eigenforms by Lanneau--Nguyen~\cite{Lanneau-Nguyen:H4}:
	\begin{thm}[{\cite[Theorem~1.1]{Lanneau-Nguyen:H4}}]
	\label{thm:Prym-classification}
        For $D \geq 17$, the locus $E_D(4)$ is non-empty if and only if $D \modn{8} 1$ or $D \modn{4} 0$. Moreover, $E_D(4)$ has exactly two connected components for $D \modn{8} 1$, and it is connected for $D \modn{4} 0$.
        
        For $D < 17$, the locus $E_D(4)$ is non-empty if and only if $D \in \{8, 12\}$ and, in these cases, it is connected.
        
        Finally, each component of $E_D(4)$ is a closed $\GL^+(2,\RR)$-orbit. Moreover, the loci $E_D(4)$ are pairwise disjoint for distinct $D$. 
	\end{thm}

\section{Prototypical representatives of Weierstrass Prym eigenforms in genus three}
\label{sec:prototypes}
    Following Lanneau--Nguyen \cite{Lanneau-Nguyen:H4}, we say that a quadruple of integers $(a,b,c,e)$ is a \emph{splitting prototype} of discriminant $D$ if:
    \begin{gather*}
        D = e^2 + 8 bc, \quad
        0 \leq a < \gcd(b,c), \quad
    	2c + e < b, \\
    	0 < b,c,
    	\quad \text{ and } \quad
    	\gcd(a,b,c,e) = 1.
    \end{gather*}

    With each splitting prototype $(a, b, c, e)$, we associate two translation surfaces, depicted in \Cref{fig:modelA+-}, where $\lambda = \frac{e + \sqrt{D}}{2}$.
    We call such surfaces \emph{prototypical surfaces} associated with the splitting prototype $(a, b, c, e)$ and denote them by $\modelA^+(a,b,c,d)$ and $\modelA^-(a,b,c,e)$; we refer to them as \emph{model $\modelA^+$} and \emph{model $\modelA^-$}, respectively. The work of Lanneau--Nguyen \cite{Lanneau-Nguyen:H4} shows that these surfaces belong to $E_D(4)$.
    
    \begin{figure}[t!]
        \centering
        \raisebox{-0.5\height}{\includegraphics[scale=0.8]{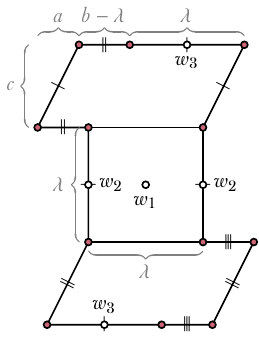}}
        \qquad\qquad
        \raisebox{-0.5\height}{\includegraphics[scale=0.8]{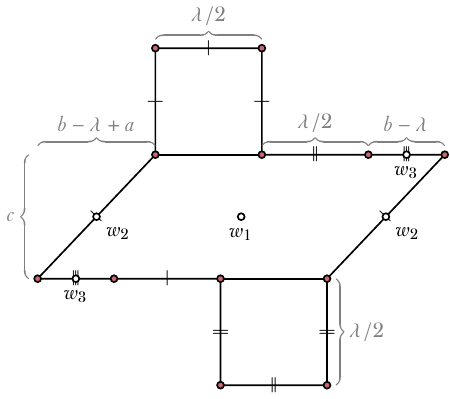}}
        \caption{Prototypical surfaces $\modelA^+(a, b, c, e)$ (left) and $\modelA^-(a, b, c, e)$ (right), with $\lambda = \frac{e + \sqrt{D}}{2}$. These examples depict $\modelA^+(1, 4, 2, -3)$ for $D = 73$, and $\modelA^-(1, 7, 3, -3)$ for $D = 177$, respectively.}
        \label{fig:modelA+-}
    \end{figure}

    Splitting prototypes correspond to some three-cylinder cusps in $E_D(4)$~\cite[Section~4.1]{Lanneau-Nguyen:H4} (see \Cref{sec:parabolic-action} for more details on cylinder decompositions).
    They do not account for all cusps. However, for each splitting prototype $(a,b,c,e)$ of discriminant $D$, the models $\modelA^+$ and $\modelA^-$ can be identified with distinct cusp of $E_D(4)$. In most cases this is enough to prove \Cref{thm:main_H(4)}. In fact, we have the following.
    \begin{thm}[{\cite[Proposition~4.7]{Lanneau-Nguyen:H4}}]
    \label{thm:no-modelA:D=8}
        Let $X \in E_D(4)$. If $X$ admits no cylinder decompositions in model $\modelA^+$ or model $\modelA^-$, then $D = 8$.
    \end{thm}
    
    \subsection{Reduced prototypes}
    A \emph{reduced} prototypes is a splitting prototype $(a,b,c,d)$ with $c = 1$. We will work exclusively with reduced prototypes. Note that these splitting prototypes necessarily have the form
    \[
        (a,b,c,e) = \left(0,\frac{D-e^2}{8},1,e\right)
    \]
    and are therefore parameterized by
    \[
        S_D = \{e \in \ZZ \ \mathbin{|}\  e^2 \modn{8} D, e^2 < D, (e+4)^2 < D\}.
    \]
    
    The associated prototypical surfaces are denoted $\modelA^+_D(e)$ and $\modelA^-_D(e)$. That is,
    \[
        \modelA^\pm_D(e) = \modelA^\pm\left(0,\frac{D-e^2}{8},1,e\right) \in E_D(4).
    \]

    \subsection{Disconnected loci}
    When $E_D(4)$ is disconnected, there is a parity invariant telling apart the two components.
    Even if it is possible to obtain a closed expression for this invariant, this is outside of the scope of our work (and there is no direct reference known by the authors).
    Indeed, we only need to tell some particular cases apart for our purposes. More precisely, we use the following:
    \begin{lem}\label{lem:spin-Prym}
        Let $D$ be an odd discriminant such that $E_D(4)$ is non-empty (and, thus, disconnected). Then,

        \begin{itemize}
            \item If $e \in S_D$, then $\modelA^+_D(e)$ and $\modelA^-_D(e)$ lie at different components of $E_D(4)$.
            \item If $e,e' \in S_D$, then $\modelA^+_D(e)$ and $\modelA^+_D(e')$ (resp.\ $\modelA^-_D(e)$ and $\modelA^-_D(e')$) lie at the same component of $E_D(4)$ if and only if $e \modn{4} e'$.
        \end{itemize}
        That is, $\modelA^+_D(e)$ and $\modelA^-_D(e')$ lie at the same component of $E_D(4)$ if and only if $e \modn{4} -e'$.
    \end{lem}
    
    \begin{proof}
        The first claim follows directly from Lanneau--Nguyen \cite[Theorem~6.1]{Lanneau-Nguyen:H4}. 
        The proof of the second claim is analogous to McMullen's work in genus two \cite[Theorem~5.3]{McMullen:discriminant_spin} (cf.\ Lanneau--Nguyen \cite[Lemma~6.2]{Lanneau-Nguyen:H4}).
    \end{proof}
    
    \subsection{Re-scaled prototypical surfaces}
    We also use, in \Cref{sec:restrictions}, a re-scaled version of the surface associated with a reduced prototype. This surface is obtained by the action of the diagonal matrix $\left(\begin{smallmatrix}2\lambda^{-1} & 0 \\ 0 & 1 \end{smallmatrix}\right) \in \GL^+(2, \RR)$ on $\modelA^-_D(e)$, with $e \in S_D$, as in \Cref{fig:Z_H(4)}. We denote such surface by $\modelZ_D(e) \in E_D(4)$.
    
    \begin{figure}[t!]
        \centering
        \includegraphics[scale=0.8]{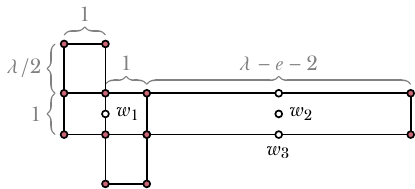}
        \caption{Re-scaled prototypical surface $\modelZ_D(e)$ obtained by applying $\left(\begin{smallmatrix} 2\lambda^{-1} & 0 \\ 0 & 1 \end{smallmatrix}\right)$ to $\modelA^-_D(e)$. This example shows $\modelZ_{116}(-6)$.}
        \label{fig:Z_H(4)}
    \end{figure}

    \section{Action of affine multitwists}
    \label{sec:parabolic}
 
    In this section, we study the permutation group induced by the action of affine multitwists that correspond to \emph{parabolic} elements of the Veech group.
    More precisely, given a reduced prototype $e \in S_D$ we show that the horizontal and vertical twists on the associated prototypical surfaces $\modelA^{\pm}(e)$ induce either trivial elements or transpositions in $G(\modelA^{\pm}(e))$. Then we compute the group generated by these affine multitwists.

    \subsection{Parabolic elements and cylinder decompositions}
    \label{sec:parabolic-action}

    In a translation surface $(X, \omega)$, every closed regular geodesic is accompanied by a collection of parallel closed regular geodesics. A cylinder on such a surface is defined as a maximal open annulus consisting of isotopic simple closed regular geodesics. Specifically, a cylinder $C$ is isometric to the product of an open interval and a circle. For a cylinder isometric to $\RR / w \ZZ \times (0, h)$, we refer to $w > 0$ as its \emph{circumference} and $h > 0$ as its \emph{height}. The \emph{modulus} of the cylinder is given by $m = \frac{w}{h} > 0$, and the \emph{core curve} is the simple closed geodesic that projects to the center of the interval. A \emph{saddle connection} is a geodesic segment connecting two singularities (possibly the same), with no singularities in its interior. The boundary of a cylinder is composed of parallel saddle connections.
    
    A \emph{cylinder decomposition} is a collection of parallel cylinders whose closures cover the entire surface, and we say that the decomposition is \emph{in direction $\theta$} if the angle between the horizontal direction and the core curve of any cylinder in the decomposition is $\theta \in [0, 2\pi)$.
    
    According to Veech's work~\cite{Veech}, there is a close relationship between cylinder decompositions and parabolic elements. Specifically, let $\theta \in [0, 2\pi)$ be a fixed angle, and let $R_\theta$ denote the matrix that rotates the plane counterclockwise by $\theta$. For $t \in \RR$, let $T_t = \left(\begin{smallmatrix} 1 & t \\ 0 & 1 \end{smallmatrix}\right)$ be the matrix that twists the horizontal direction by $t$.
    In this context, a parabolic matrix written as a product $P_\theta^t = R_\theta T_t R_\theta^{-1}$ belongs to the Veech group $\SL(X)$ if and only if there exists a cylinder decomposition $\{C_i\}_{i \in I}$ of $X$ in direction $\theta$ satisfying that the moduli $m_i$ of the cylinders $C_i$ are rationally related: that is $r_{ij} = \frac{m_i}{m_j} \in \QQ$ for all $i, j \in I$. In particular, for any $t \in \RR$ that is an integer multiple of the modulus of each cylinder in the decomposition, that is, such that for each $i \in I$ there exists $k_i \in \ZZ$ with $t = k_i m_i$, then $P_\theta^t \in \SL(X)$.
    In this case, the corresponding affine transformation preserves the cylinder decomposition by twisting each cylinder $C_i$ exactly $k_i$ times along itself. In particular, $P_\theta^t$ acts on each core curve $\gamma_i$ by $k_i$ half-turns, that is, by a rotation of angle $k_i \pi$. The $k_i$ can be chosen to be all positive and minimal, and in this case the corresponding affine transformation is referred to as the \emph{twist} in direction $\theta$, and we say that $\theta$ is a \emph{parabolic direction}.

    Thanks to the \emph{Veech dichotomy}~\cite{Veech}, parabolic directions are well understood in the case of Veech surfaces. In fact, a direction $\theta \in [0, 2\pi)$ is parabolic if and only if there is either a saddle connection or a closed geodesic in direction $\theta$.

    \subsection{Prototypical cylinder decompositions in $E_D(4)$}
    We are interested in the horizontal and the vertical cylinder decompositions on $\modelA^\pm(e)$, for reduced prototypes $e \in S_D$.
    In every such case, we get a three-cylinder decomposition with two isometric cylinders. Thus, it is enough to consider the moduli of two non-isometric cylinders, say $m_1$ and $m_2$, in such decomposition in order to understand the action of the corresponding twist in such direction, say $\theta \in \{0,\pi\}$.

    Then, we can write
    \[
        r = \frac{m_1}{m_2} = \frac{p}{q} \in \QQ,
    \]
    where $p$ and $q$ are positive coprime integers. Then, the corresponding twist is described by taking $k_1 = q$ and $k_2 = p$. In other words, if $t = qm_1 = pm_2$, then the parabolic element $P_\theta^t$ belongs to $\SL(X)$, and $t > 0$ is minimal for this property. Since $t$ will be always chosen in this way, we can omit it from the notation and define $r_\theta = \frac{m_1}{m_2} = \frac{p}{q}$.

    By the previous discussion, we can compute the action of $P_\theta^t$ at the level of (regular) Prym fixed points, depending on the parity of the numbers $p$ and $q$.
    In fact, if two Prym fixed points lie on $\gamma_1$, then they are exchanged by $P_\theta^t$ if and only if $p$ is odd. Similarly, if they lie on $\gamma_2$, they are exchanged if and only if $q$ is odd.
    Here $\gamma_i$ is the core curve of the cylinder $C_i$ of modulus $m_i$, $i=1,2$.
    
    \noindent\textbf{Notation.}
    In what follows we denote
    \begin{align*}
        \frac{\text{odd}}{\text{odd}} &= \left\{\frac{p}{q} \in \QQ \ \mathbin{\Big|}\  p, q \in \NN, \gcd(p,q) = 1, p \modn{2} q \modn{2} 1\right\}, \\
        \frac{\text{even}}{\text{odd}} &= \left\{\frac{p}{q} \in \QQ \ \mathbin{\Big|}\  p, q \in \NN, \gcd(p,q) = 1, p \modn{2} 0, q \modn{2} 1\right\} \qquad \text{and} \\
        \frac{\text{odd}}{\text{even}} &= \left\{\frac{p}{q} \in \QQ \ \mathbin{\Big|}\  p, q \in \NN, \gcd(p,q) = 1, p \modn{2} 1, q \modn{2} 0\right\}.
    \end{align*}

    Additionally, we will always choose $C_1$ to be a cylinder with largest circumference for any model-$\modelA$ decomposition of a surface in $E_D(4)$. Accordingly, we refer to $C_1$ as a \emph{long cylinder} and $C_2$, as a \emph{short cylinder}.
    
    \subsection{Horizontal and vertical twists on reduced prototypes}
    \label{sec:parabolic_H(4)}
    Given a reduced prototype $e \in S_D$ of discriminant $D$, a quick examination of \Cref{fig:modelA+-,fig:Z_H(4)} shows that the horizontal and vertical twists on the associated prototypical surfaces $\modelA^{\pm}(e)$ induce either trivial elements or transpositions in $G(\modelA^{\pm}(e))$. The following lemma explicitly states the conditions so that Prym fixed points are exchanged or not.
    \begin{lem} \label{lem:horizontal_vertical_H(4)}
        Let $D > 4$ be a real quadratic discriminant such that $E_D(4)$ is non-empty and let $e \in S_D$. Let $b = \frac{D - e^2}{8}$. We have that:
        \begin{itemize}
            \item the horizontal twist acts on the Prym fixed points of $\modelA^{\pm}_D(e)$ either trivially or by the transposition $(1\;2)$. In the case of $\modelA^+_D(e)$, the action is non-trivial if and only if $b$ is odd, while it is always non-trivial for $\modelA^-_D(e)$; and
            \item the vertical twist acts on the Prym fixed points of $\modelA^{\pm}_D(e)$ either trivially or by the transposition $(1\;3)$. In the case of $\modelA^+_D(e)$ the action is always non-trivial, while it is non-trivial for $\modelA^-_D(e)$ if and only if $b - e - 2$ is odd.
        \end{itemize}
    \end{lem}
    
    \begin{proof}
        The strategy of the proof is simply to analyse the cylinder decompositions and the ratio of the moduli of the resulting cylinders for the horizontal and vertical directions (c.f.~\cite[Section~4.3]{Gutierrez-Romo--Pardo:H2}).
        
        For the horizontal direction, the prototypical surface $\modelA^+_D(e)$ has two long cylinders of moduli $m_1 = b$, and one short cylinder of modulus $m_2 = 1$. The Prym fixed points $w_1$ and $w_2$ lie in the interior of the short cylinder. See \Cref{fig:modelA+-} (left). Therefore, the corresponding ratio is $r_{\mathrm{h}} = \frac{m_1}{m_2} = b$. Thus, the points are exchanged if and only if $r_{\mathrm{h}} \in \frac{\text{odd}}{\text{odd}} \sqcup \frac{\text{odd}}{\text{even}}$, which clearly happens if and only if $r_{\mathrm{h}}=b$ is odd, as it is an integer.
        
        Moreover, the prototypical surface $\modelA^-_D(e)$ has one long cylinder of modulus $m_1 = b$, and two short cylinders of moduli $m_2 = 1$. The Prym fixed points $w_1$ and $w_2$ lie in the interior of the long cylinder.
        See \Cref{fig:modelA+-} (right). We obtain that the corresponding ratio is $r_{\mathrm{h}} = \frac{m_1}{m_2} = b$ as in the previous case. Thus, the points are exchanged if and only if $r_{\mathrm{h}} \in \frac{\text{odd}}{\text{odd}} \sqcup \frac{\text{even}}{\text{odd}}$, which is always the case as $r_{\mathrm{h}}=b$ is an integer.
        
        The vertical direction can be treated similary. Indeed, observe that the prototypical surface $\modelA^+_D(e)$ has a long cylinder of modulus $m_1 = \frac{\lambda + 2}{\lambda}$, and two short cylinders of moduli $m_2 = \frac{1}{b - \lambda}$, in the vertical direction. The points $w_1$ and $w_3$ lie in the interior of the long cylinder. The corresponding ratio is
        \[
            r_{\mathrm{v}} = \frac{m_1}{m_2} = \frac{\frac{\lambda + 2}{\lambda}}{\frac{1}{b - \lambda}} = \frac{D - (e + 4)^2}{8} = b - e - 2
        \]
        which is also an integer so the points are exchanged unconditionally.
        
        Finally, the prototypical surface $\modelA^+_D(e)$ has two long cylinders of modulus $m_1 = \frac{\lambda/2 + 1}{\lambda/2}$ and a short cylinder of modulus $m_2 = \frac{1}{b - \lambda}$, in the vertical direction. The points $w_1$ and $w_3$ lie in the interior of the short cylinder. Thus, $r_{\mathrm{v}} = \frac{m_1}{m_2} = b - e - 2$, so these points are exchanged by the vertical twist if and only if $b - e - 2$ is odd.
    \end{proof}
    
    \begin{rem}
    \label{rem:butterfly_H(4)}
        To keep the exposition straightforward, we have omitted the concept of \emph{butterfly moves}. However, one can interpret the cylinder decomposition in the vertical direction on $\modelA^{\pm}_D(e)$ as the the result of a butterfly move $B_\infty$, transforming it into the canonical surface $\modelA^{\mp}_D(-e - 4)$ (see \cite[Proposition 7.5(2)]{Lanneau-Nguyen:H4}). This gives an alternative explanation on the dependence of the action induced by the vertical twist on the parity of $\frac{D - (e + 4)^2}{8}$.
    \end{rem}

    The previous lemma already allows us to show that $G(X)$ is non-trivial in most cases.
    \begin{prop}\label{prop:transposition}
        Let $D \neq 8$ be a real quadratic discriminant such that $E_D(4)$ is non empty and let $X \in E_D(4)$. Then, $G(X)$ contains a transposition, which is given by the action of a parabolic element in $\SL(X)$.
    \end{prop}
    \begin{proof}
        By \Cref{thm:no-modelA:D=8}, for any $D$ as in the statement there exists a model-$\modelA^{\pm}$ cylinder decomposition. In particular, $S_D \neq \emptyset$.
        
        Now, given a reduced prototype $e \in S_D$, by \Cref{lem:horizontal_vertical_H(4)}, the horizontal twist on $\modelA^-_D(e)$ gives $(1\;2) \in G(\modelA^-_D(e))$ and the vertical twist on $\modelA^+(e)$ gives $(1\;3) \in G(\modelA^+_D(e))$.
        
        If $D$ is even, $E_D(4)$ is connected, so this already proves the desired result for every $X \in E_D(4)$.
        On the other hand, if $D$ is odd and $E_D(4)$ is disconnected, \Cref{thm:Prym-classification} shows that $\modelA^-_D(e)$ and $\modelA^+(e)$ lie on different component, so the result also follows for any $X \in E_D(4)$.
    \end{proof}
    
    \begin{rem}\label{rem:Sym3}
        By \Cref{lem:horizontal_vertical_H(4)}, in order to get the entire symmetric group $\Sym_3$ from the horizontal and vertical twists on $\modelA^+_D(e)$, we need that $r_{\mathrm{h}} = b = \frac{D - e^2}{8}$ is odd, or, equivalently, that $D \modn{16} 8 + e^2$. Similarly, on $\modelA^-_D(e)$, we need $r_{\mathrm{v}} = \frac{D - (e+4)^2}{8}$ to be odd, that is, we need that $D \modn{16} 8 + (e+4)^2$.
        Since \Cref{lem:spin-Prym} shows that $\modelA^+(e)$ and $\modelA^-(-e-4)$ lie on the same component, verifying these conditions for either case is actually equivalent.
    \end{rem}
    
    The next series of lemmas establish \Cref{thm:main_H(4)} for $D \geq 12$ satisfying $D \modn{4} 0$ or $D \modn{8} 1$, with $D \notin \{25, 33\}$, and such that $D$ is either odd or an even quadratic nonresidue modulo 16. More precisely, we will show that $G(X) = \Sym_3$ for any $X \in E_D(4)$ for such $D$. Observe that if $D$ is even, it is a quadratic nonresidue if and only if $D \modn{16} 8$ or $D \modn{16} 12$. We will analyse these cases separately.
    
    \begin{lem}
        \label{lem:D=8 mod 16_H(4)}
        Let $D > 0$ be a real quadratic discriminant with $D \modn{16} 8$ and $X \in E_D(4)$. If $D \neq 8$, then $G(X) = \Sym_3$ and is generated by the action of two parabolic elements in $\SL(X)$.
    \end{lem}
    \begin{proof}
        Clearly, $D \modn{8} 0$. By hypothesis, $D \geq 24$, so we have that $0 \in S_D$. 
        Since $E_D(4)$ is connected, it is enough to prove the result for the prototypical surface $\modelA^+_D(0)$. As discussed in \Cref{rem:Sym3}, it suffices to show that $D \modn{16} 8 + 0^2$, which also holds by hypothesis.
    \end{proof}
    
    \begin{rem}
        Observe that the previous proof fails for $D = 8$, as $0 \notin S_8 = \emptyset$. We postpone the proof for the case $D = 8$ to \Cref{sec:D=8_H(4)}.
    \end{rem}
    
    \begin{lem}
        \label{lem:D=12 mod 16_H(4)}
        Let $D > 0$ be a real quadratic discriminant with $D \modn{16} 12$ and let $X \in E_D(4)$. Then, $G(X) = \Sym_3$ and is generated by the action of two parabolic elements in $\SL(X)$.
    \end{lem}
    \begin{proof}
        We have that $D \modn{8} 4$. By hypothesis, $D \geq 12$, so we have that $-2 \in S_D$. 
        Since $E_D(4)$ is connected, it is enough to prove the result for the prototypical surface $\modelA^+_D(-2)$. As discussed in \Cref{rem:Sym3}, it suffices to show that $D \modn{16} 8 + (-2)^2$, which also holds by hypothesis.
    \end{proof}
    
    \begin{lem}
        \label{lem:D=1 mod 8_H(4)}
        Let $D > 0$ be a real quadratic discriminant with $D \modn{8} 1$ and let $X \in E_D(4)$. If $D \geq 33$, then $G(X) = \Sym_3$ and is generated by the action of two parabolic elements in $\SL(X)$.
    \end{lem}
    \begin{proof}
        Assume first that $D \modn{16} 1$. By hypothesis, $D \geq 33$, so we have that $\{-5, -3\} \subseteq S_D$. Observe that $D \modn{16} 8 + e^2$ for $e \in \{-5, -3\}$, so, as discussed in \Cref{rem:Sym3}, we obtain directly that $G(\modelA^+_D(-5)) = \Sym_3$ and that $G(\modelA^+_D(-3)) = \Sym_3$. Since, by \Cref{lem:spin-Prym}, these surfaces lie on different components, we obtain the result for any $X \in E_D(4)$.
        
        Now, assume that $D \modn{16} 9$. By hypothesis, $D \geq 41$ and, therefore, we have that $\{-1, 1\} \subseteq S_D$. In this case, $D \modn{16} 8 + e^2$ for $e \in \{-1, 1\}$. The result follows analogously.
    \end{proof}
    
    \begin{rem}
        Observe that the previous proof fails for $D \in \{17, 25\}$. We postpone the proof for these cases to \Cref{sec:D=17_H(4),sec:D=25_H(4)}.
    \end{rem}

    \section{Restrictions for permutations}
    \label{sec:restrictions}
    In the cases where $D \nmodn{16} 8+e^2$ for every $e \in S_D$, the arguments in the previous section do not allows us to obtain the full symmetric group $\Sym_3$ and, in fact, generating this group is not possible.
    More precisely, we shall see that, in these cases, the group of permutations of Prym fixed points consists only on the identity and a transposition.
    We follow a similar strategy to that in \cite[Section~5]{Gutierrez-Romo--Pardo:H2} by imposing restrictions for a matrix to act on the Prym fixed points in terms of their relative position on a given surface. Note that $D \modn{16} e^2$ for every $e \in S_D$ if and only if $D \modn{16} 0$ or $4$, so the remaining cases are precisely those where $D$ is a quadratic residue modulo $16$.
    
    \subsection{Arithmetic case}
    \label{sec:arithmetic}

    Let $D = d^2$, where $d > 4$ is an even integer. The Prym locus $E_{d^2}(4)$ consists of a single $\GL^+(2, \RR)$-orbit of square-tiled surfaces. A (primitive) \emph{square-tiled surface} is a translation surface $(X, \omega)$ such that the set of periods of $\omega$ is the lattice $\ZZ[i] \subset \CC$. In other words, integrating $\omega$ along any saddle connection results in an element of $\ZZ[i]$. For such a surface, integrating $\omega$ yields a holomorphic map
    \[
        p \colon X \to E = \CC / \ZZ[i],
    \]
    which can be normalized so that it is branched only over $z = 0$.

    The Prym involution descends to involution $z \mapsto -z$ on $E$ and therefore, a Prym fixed point $w \in X$ can only project to a $2$-torsion point on the unit torus.
    The HLK-invariant is the combinatorial structure of how the Prym fixed point projects onto the square torus, and it imposes constraints on how these points can be permuted by the affine group, since affine transformations commute with the projection.
    
    Let $e \in S_{d^2}$ and consider the re-scaled prototypical surface $X = \modelZ_D(e)$. This surface is a primitive square-tiled surface as in \Cref{fig:rep-d2}.
    Moreover, since $d^2 \modn{16} e^2$, we have that $d \modn{4} e$.
    It follows that $\lambda - e = \frac{d-e}{2}$ is even and, therefore, that $w_1$ and $w_2$ project down to $v = i/2 \in \CC/\ZZ[i]$, and that $w_3$ projects down to $0 \in \CC/\ZZ[i]$.
    Thus, it is not possible to exchange $w_3$ with $w_1$ or $w_2$.
    By \Cref{prop:transposition}, we get that $G(X) \cong \Sym_2$.
        
    \begin{figure}[t!]
        \centering
        \includegraphics[scale=1]{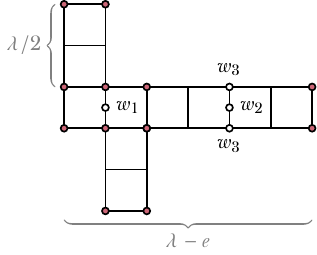}
        \caption{The square-tiled surface $\modelZ_{d^2}(e) \in E_{d^2}(4)$, for even $d>4$.
        This example shows $\modelZ_{100}(-2)$.}
        \label{fig:rep-d2}
    \end{figure}
    
    Since the locus $E_D(4)$ is connected for even $D$, we get the following:
        
    \begin{prop}
    \label{prop:arith-d2_H(4)}
        Let $d > 4$ be an even integer and let $X \in E_{d^2}(4)$. Then, $G(X) \cong \Sym_2$. Moreover, $G(X)$ is generated by the action of a single parabolic element in $\SL(X)$.
        \qed
    \end{prop}
    
    \subsection{Non-arithmetic case}
    \label{sec:non-arithmetic}
    Following ideas from \cite[Section~5.2]{Gutierrez-Romo--Pardo:H2} in the genus two case, we reduce the non-arithmetic case to the arithmetic case. The computations in the genus-three case are much more subtle, so we need some simple results before we are able to state the reduction.
    For this, consider first the auxiliary set
    \[
        R_D = \{e \in \ZZ \ \mathbin{|}\  e \modn{2} D, e^2 < D, (e+2)^2 < D\}.
    \]
    This set is in fact meaningful. It plays the role of reduced prototypes in the genus~$2$ case (see, e.g., \cite[Section~3.2.2]{Gutierrez-Romo--Pardo:H2}).
    
    \begin{lem}
        Let $D > 16$ be an even real quadratic discriminant such that $E_D(4)$ is non-empty. Then, $\frac{1}{2} S_D = R_{D/4}$.
    \end{lem}
    \begin{proof}
        Recall first that $D \modn{4} 0$, so $D/4$ is indeed an integer. Likewise, every element of $S_D$ is even, so $\frac{1}{2} S_D$ consists entirely of integers.
        
        Observe that each condition in the definition of $\frac{1}{2}S_D$ translates to a condition in the definition of $R_{D/4}$. More precisely, given $e \in S_D$, we have that $e^2 \modn{8} D$ if and only if $\left(\frac{e}{2}\right)^2 \modn{2} \frac{e}{2} \modn{2} \frac{D}{4}$. Moreover, the condition $e^2 < D$ is equivalent to $\left(\frac{e}{2}\right)^2 < \frac{D}{4}$. Finally, $(e + 4)^2 < D$ holds if and only if $\left(\frac{e}{2} + 2\right)^2 < \frac{D}{4}$.
    \end{proof}
    
    If $D > 16$ is an even quadratic discriminant and $e \in S_D$, then we know that $\ZZ[\lambda] = \O_D$. The previous lemma shows that, moreover, $\ZZ[\lambda/2] = \O_{D/4}$. Indeed, we have that
    \[
        \frac{\lambda}{2} = \frac{e + \sqrt{D}}{4} = \frac{e/2 + \sqrt{D/4}}{2}
    \]
    and $\frac{e}{2} \in R_D$. This can be extended to other generators of $\O_D$ as shown by the next lemma:
    
    \begin{lem}
        Let $D > 0$ be an even real quadratic discriminant and assume that it is a quadratic residue modulo $16$, say $D \modn{16} d^2$ for some $d \in \ZZ$. Take $\rho = \frac{\sqrt{D} - d}{2}$. Then, we have that $\ZZ[\rho/2] = \O_{D/4}$.
    \end{lem}
    \begin{proof}
        If $e \in S_D$ and $\lambda = \frac{e + \sqrt{D}}{2}$, we have that $\rho = \lambda - \frac{e + d}{2}$. Note that $e + d \modn{4} 0$. Indeed, we have that $D \modn{16} d^2$ and $D \modn{8} e^2$, which implies that $e^2 \modn{8} d^2$. Since both $e$ and $d$ are even, this implies that $e \modn{4} d$, and, again using that they are both even, we obtain that $e + d \modn{4} 0$. Thus,
        \[
            \frac{\rho}{2} = \frac{\lambda}{2} - \frac{e + d}{4} = \frac{e/2 + \sqrt{D/4}}{2} - \frac{e + d}{4} = \frac{\lambda}{2} - \frac{e + d}{4}
        \]
        Since $\frac{e + d}{4}$ an integer, we conclude that $\ZZ[\rho/2] = \O_{D/4}$.
    \end{proof}

    \begin{lem}
        Let $D > 4$ be even and such that $E_D(4)$ is non-empty. Let $e \in S_D$ and let $A \in \SL(\modelZ_D(e))$. Then, $A = \Big(\begin{smallmatrix} \alpha & \beta \\ \gamma & \delta \end{smallmatrix}\Big)$ for some $\alpha, \beta \in \O_D$ and $\gamma, \delta \in \O_{D/4}$.
    \end{lem}
    \begin{proof}
        Fix $A \in \SL(\modelZ_D(e))$. We proceed as in the proof \cite[Lemma~5.4]{Gutierrez-Romo--Pardo:H2}, which shows that each column of $A$ is a holonomy vector of a saddle connection. As the horizontal sides and the horizontal coordinates of translation vectors of $\modelZ_D(e)$ belong to $\O_D$, while the vertical sides and vertical coordinates of translation vectors of $\modelZ_D(e)$ belong to $\O_{D/4}$, we get the desired result.
    \end{proof}

    In the following, we focus on the relative positions with respect to the singularity of the regular Prym fixed points in $\modelZ_D(e)$. In particular, we are interested in elements in $\O_{D/4} \times \frac{1}{2}\O_D$.
    In fact, if $v_i$ is the corresponding vector for the regular Prym fixed point $w_i$, $i \in \{1,2,3\}$, then $v_i \in \O_{D/4} \times \frac{1}{2}\O_D$ (modulo $\O_D \times \O_{D/4}$).
    
    For the rest of this section, we assume that $D$ is a quadratic residue modulo $16$, say $D \modn{16} d^2$. We also fix $\rho = \frac{\sqrt{D} - d}{2}$, so $\O_D = \ZZ[\rho]$.
    Each element $x$ of the field $\QQ(\sqrt{D})$ can be uniquely expressed as $x = p + q\rho$ with $p, q \in \QQ$. We refer to $p$ as the \emph{rational part} of $x$ (note that this depends on the choice of $\rho$). Additionally, we define $\fr(x)$ as the fractional part of the rational part of $x$. Namely, $\fr(x) = \{p\} = (p \bmod{1})$.

    \noindent\textbf{Notation.}
    \begin{itemize}[beginpenalty=10000,endpenalty=10000]
        \item For $A = \Big(\begin{smallmatrix} \alpha & \beta \\ \gamma & \delta \end{smallmatrix}\Big)$, with $\alpha, \beta \in \O_D$ and $\gamma, \delta \in \O_{D/4}$, we write
        \begin{itemize}
            \item $\eta = p_\eta + q_\eta\rho \in \O_D$, where $p_\eta, q_\eta \in \ZZ$ for each $\eta \in \{\alpha, \beta\}$; and
            \item $\eta = p_\eta + q_\eta \frac{\rho}{2} \in \O_{D/4}$, where $p_\eta, q_\eta \in \ZZ$ for each $\eta \in \{\gamma, \delta\}$. 
        \end{itemize}
        We also write $P_A = \Big(\begin{smallmatrix} p_\alpha & p_\beta \\ p_\gamma & p_\delta \end{smallmatrix}\Big)$ and $Q_A = \Big(\begin{smallmatrix} 2q_\alpha & 2q_\beta \\ q_\gamma & q_\delta \end{smallmatrix}\Big)$, so $A = P_A + Q_A \frac{\rho}{2}$. Observe that both $P_A$ and $Q_A$ have integer coefficients and, moreover, that $\det(Q_A) \modn{2} 0$.
        \item For $v \in \O_{D/4} \times \frac{1}{2}\O_D$, we write $v = \left(p_1 + q_1\frac{\rho}{2}, \frac{p_2}{2} + \frac{q_2}{2} \rho\right)$, where $p_i, q_i \in \ZZ$ for each $i \in \{1,2\}$. We write $p_v = (2p_1, p_2)$ and $q_v = (q_1, q_2)$ as well, so we have $v = \frac{1}{2}(p_v + q_v \rho)$.
    \end{itemize}
    Then, 
    \begin{align} 
        A v &= \frac{1}{2} \left(P_A + Q_A \frac{\rho}{2}\right)(p_v + q_v \rho) = \frac{1}{2} \left(P_A p_v + Q_A q_v \frac{\rho^2}{2} + (P_A q_v + Q_A p_v) \frac{\rho}{2}\right) \notag \\
        &= \frac{1}{2}P_A p_v + Q_A q_v \left(\frac{D - d^2}{16} - \frac{d \rho}{4}\right) + (P_A q_v + Q_A p_v) \frac{\rho}{4} \notag \\
        &= \frac{1}{2}P_A p_v + \frac{D - d^2}{16} Q_A q_v + \left(P_A q_v + Q_A p_v - d Q_A q_v\right) \frac{\rho}{4} \label{eq:Av_H(4)}
    \end{align}

    \begin{lem}
        Let $D > 0$ be a real quadratic discriminant and assume that it is a quadratic residue modulo $16$, say $D \modn{16} d^2$ for some $d \in \ZZ$, and let $e \in S_D$ and $A \in \SL(\modelZ_D(e))$. Take $\rho = \frac{\sqrt{D} - d}{2}$. Then, we have $\fr(A v) = \frac{1}{2} P_A p_v \bmod{1}$ for any $v \in \O_{D/4} \times \left(\frac{1}{2}\O_D\right)$. 
        Moreover, $P_A$ is non-singular modulo~$2$.
    \end{lem}
    
    \begin{proof}
        By \cref{eq:Av_H(4)}, the rational part of $A v$ is
        \[
            \frac{1}{2} P_A p_v + \frac{D - d^2}{16} Q_A q_v
        \]
        Since $\frac{D - d^2}{16} \in \ZZ$, we get $\fr(A v) = \frac{1}{2} P_A p_v \bmod 1$.
        
        On the other hand,
        \[
            1 = \det(A) = \det\left(P_A + Q_A \frac{\rho}{2}\right) = \det(P_A) + \det(Q_A) \frac{\rho^2}{4} + \begin{vmatrix}
                p_\alpha & p_\beta \\ q_\gamma & q_\delta
            \end{vmatrix} \frac{\rho}{2} + \begin{vmatrix}
                q_\alpha & q_\beta \\ p_\gamma & p_\delta
            \end{vmatrix} \rho.
        \]
        Taking rational parts in the previous equality yields $\det(P_A) + \frac{D - d^2}{16} \det(Q_A) = 1$.
        Thus, $\det(P_A) \modn{2} 1$ since $\det(Q_A) \modn{2} 0$.
    \end{proof}

    The previous lemmas show that, when $D \modn{16} d^2$, the action of $\SL(\modelZ_D(e))$ on Prym fixed points behaves as in the arithmetic case treated in \Cref{sec:arithmetic}.
    
    A straightforward computation shows that the fractional part of the rational part of the vectors $v_i \in \O_{D/4} \times \frac{1}{2}\O_D$ on the surface $\modelZ_D(e)$, for $e \in S_D$, are
    \[
        \fr(v_1) = (0, 1/2), \quad
        \fr(v_2) = (0, 1/2)
    	\quad \text{ and } \quad
    	\fr(v_3) = (0, 0),
    \]
    when $D$ is an even quadratic residue modulo~$16$.
    See \Cref{fig:Z_H(4)}.
    It follows that
    \[
        G(\modelZ_D(e)) \leqslant \Sym(\{1, 2\}).
    \]
    
    Thus, we get the following result, analogous to \Cref{prop:arith-d2_H(4)} in the arithmetic case.
    
    \begin{cor}\label{cor:upper-bounds_H(4)}
    Let $D \geq 20$ be an even quadratic residue modulo $16$ and $X \in E_D(4)$. Then $G(X)$ is conjugate to $\Sym_2$. Moreover, $G(X)$ is generated by the action of a single parabolic element in $\SL(X)$.
    \qed
    \end{cor}
    
    \section{Weierstrass Prym eigenforms in genus three of small discriminant}
    \label{sec:remaining}
        In this section, we address the remaining cases not covered by our previous arguments.
        Namely, we deal with the cases where $D \in \{8, 17, 25\}$.
        
    \subsection{Discriminant 8}  
        \label{sec:D=8_H(4)}
        By \Cref{thm:no-modelA:D=8}, there are no model-A cusps on $E_8(4)$. However, we have the representative $S = \mathrm{B}_8(0) \in E_8(4)$, with the corresponding Prym fixed points as shown in \Cref{fig:8}, as is shown explicitly by Lanneau--Nguyen \cite[Proposition~4.7 and Figure~6]{Lanneau-Nguyen:H4}.
    \begin{figure}[t!]
        \centering
        \includegraphics[scale=0.8]{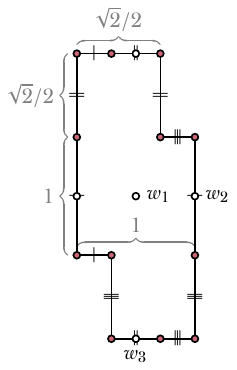}
        \caption{The surface $\mathrm{B}_8(0) \in E_8(4)$ and its Prym fixed points.}
        \label{fig:8}
    \end{figure}
        
    Since the moduli of all the horizontal cylinders in $S$ are $1$, we get that $(1\;2) \in G(S)$. Similarly, the moduli of all the vertical cylinders are $1$ as well, and we get $(1\;3) \in G(S)$.
    
    Thus, $G(S) = \Sym_3$ and, since $E_8(4)$ is connected, we get that $G(X) = \Sym_3$ for every $X \in E_8(4)$.
    \qed
    
    \subsection{Discriminant 17}
        \label{sec:D=17_H(4)} We now assume that $D = 17$, so $S_{17} = \{-3,-1\}$.
        If we take $e = -3$, we get that $r = \frac{D-e^2}{8} = 1$ is odd and, by \Cref{rem:Sym3}, we conclude that $G(\modelA^+_{17}(-3)) = \Sym_3$.
        
        On the other hand, consider the re-scaled prototypical surface $\modelZ_{17}(-3)$ in \Cref{fig:17}. Recall that this surface lies on the same component as $\modelA^-_{17}(-3)$ (since it is obtained by re-scaling this latter surface). It is then enough to exhibit an explicit parabolic element exchanging $w_2$ and $w_3$, since, by \Cref{lem:horizontal_vertical_H(4)}, the horizontal twist exchanges $w_1$ and $w_2$.
        \begin{figure}[t!]
            \centering
            \includegraphics[scale=0.8]{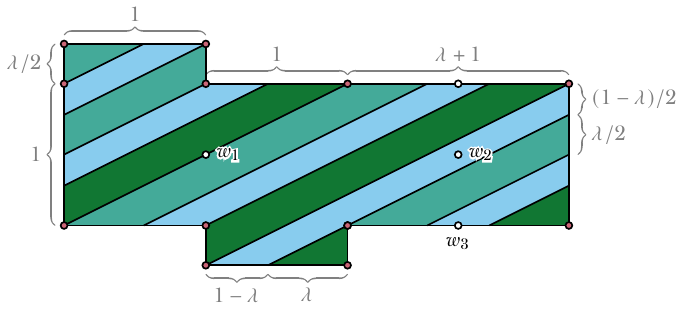}
            \caption{Re-scaled prototypical surface $\modelZ_{17}(-3)$ and the cylinder decomposition induced by the direction $\arctan(1/2)$.}
            \label{fig:17}
        \end{figure}
        
        The cylinder decomposition along $\theta = \arctan(1/2)$ (or along the direction $(2, 1)$) of $\modelZ_{17}(-3)$ is shown in \Cref{fig:17}. As explained in~\cite[Section~6.2]{Gutierrez-Romo--Pardo:H2}, it is enough to compute the ratio between the \emph{horizontal circumference} $w$ of a cylinder and its \emph{horizontal height} $h$. Let $C_1$ be the ``narrow'' cylinder containing $w_2$ and $w_3$ in its core curve, and let $C_2$ be any of the two other cylinders (since they have equal moduli). 
        
        We have that the modulus of the cylinder $C_1$ is
        \[
            m_1 = \frac{(\lambda+1) + 1 + 2 + \lambda + (2-\lambda) + \lambda}{1-\lambda} = \frac{2(\lambda + 3)}{1 - \lambda}
        \]
        and that the modulus of the cylinder $C_2$ is
        \[
            m_2 = \frac{(\lambda+1) + 1 + 2}{\lambda} = \frac{\lambda + 4}{\lambda}.
        \]
        Thus, the ratio of the moduli is
        \[
            \frac{m_1}{m_2} = \frac{2 \lambda(\lambda + 3)}{(1 - \lambda)(\lambda + 4)}.
        \]
        Since $\lambda = \frac{\sqrt{17} - 3}{2}$, a straightforward computation then shows that $\frac{m_1}{m_2} = 2$, so  we conclude that $G(\modelZ_{17}(-3)) = \Sym_3$.
        \qed
        
   \subsection{Discriminant 25} \label{sec:D=25_H(4)}
        \begin{figure}[t!]
            \centering
            \includegraphics[scale=0.8]{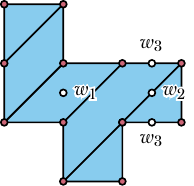}
            \caption{For $D = 25$, $S = \modelA^-_{25}(-1) \in E_{25}(4)$ is a square-tiled surface and the diagonal direction gives the permutation $(1\;3) \in G(S)$.
            }
            \label{fig:25}
        \end{figure}
        We start by fixing $D = 25$, so $S_{25} = \{-3,-1\}$.
        Taking $e = -3$, we get that $r = \frac{D-(e+4)^2}{8} = 3$ is odd. Thus, by \Cref{rem:Sym3}, $G(\modelA^-_{25}(-3)) = \Sym_3$.
        
        On the other hand, $S = \modelA^-_{25}(-1)$ is the primitive square-tiled surface in \Cref{fig:25}.
        By \Cref{lem:horizontal_vertical_H(4)}, we have $(1\;2) \in G(S)$.
        On the other hand, the diagonal direction (with angle $\pi/4$) on $S$ gives a one-cylinder decomposition whose core curve passes through $w_1$ and $w_3$.
        It follows that $(1\;3) \in G(S)$. Thus, $G(S) = \Sym_3$.
        
        Finally, by \Cref{lem:spin-Prym}, we know that $\modelA^-_{25}(-1)$ and $\modelA^-_{25}(-3)$ lie on different components, showing that $G(X) = \Sym_3$ for every $X \in E_{25}(4)$.
        \qed
    
    \section{Classification of permutations in genus three}
    \label{sec:proof} We now summarise the results that give the proof of \Cref{thm:main_H(4)}.
    
    \begin{proof}[Proof of \Cref{thm:main_H(4)}]
        When $D > 8$ is odd or a quadratic nonresidue modulo $16$,
        \Cref{lem:D=8 mod 16_H(4),lem:D=12 mod 16_H(4),lem:D=1 mod 8_H(4)} show that $G(X) = \Sym_3$, for every $X \in E_D(4)$.
        
        When $D > 0$ is an even square, \Cref{prop:arith-d2_H(4)} shows that $G(X) \cong \Sym_2$, for every $X \in E_D(4)$.
        In the non-arithmetic case, when $D$ is an even quadratic residue modulo $16$, \Cref{cor:upper-bounds_H(4)} shows that $G(X) \cong \Sym_2$, for every $X \in E_D(4)$.
        
        The exceptional cases are treated separately: $D = 8$ is treated in \Cref{sec:D=8_H(4)}, $D = 17$ is treated in \Cref{sec:D=17_H(4)}, and $D = 25$ is treated in \Cref{sec:D=25_H(4)}. Thus, $G(X) = \Sym_3$ for every $X \in E_D(4)$, for $D \in \{8,17,25\}$.
    \end{proof}

    \appendix
    
    \section{One-cylinder and two-cylinder decompositions}
    \label{app:cylinders}
    
    In this appendix, we extend the near-complete classification of one-cylinder and two-cylinder decompositions for Weierstrass Prym eigenforms in genus three by Lanneau--Nguyen \cite[Appendix~4]{Lanneau-Nguyen:H4}. Indeed, they show that these decompositions can only arise in the arithmetic case (by referencing a classical result of Thurston \cite{Thurston:diffeomorphisms}), and construct suitable prototypes to classify all possible one-cylinder and two-cylinder cusps. To complete this classification, we use the HLK-invariant to show that the surfaces having one-cylinder and two-cylinder decompositions lie in different components of $E_D(4)$ when $D \modn{8} 1$. Moreover, we give alternative elementary proofs to some already-known facts.
    
    Given an arithmetic Weierstrass Prym eigenforms in genus three, the HLK-invariant consists on the number of integral regular Prym fixed points, together with the (unordered) list of number of regular Prym fixed points of \emph{type} $h$, $v$ and $c$: points projecting to $1/2$, $i/2$ and $(1+i)/2$ on the unit torus, respectively. We will use square brackets to denote unordered lists.
    
    We will show the following:
    
    \begin{figure}[t!]
        \centering
        \begin{subfigure}{\textwidth}
        \centering
            \includegraphics[scale=0.8]{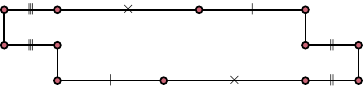}
            \caption{Model C.}
            \label{fig:modelC}
        \end{subfigure}
        \begin{subfigure}{\textwidth}
            \centering
            \includegraphics[scale=0.8]{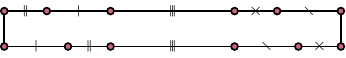}
            \caption{Model D.}
            \label{fig:modelD}
        \end{subfigure}
        \begin{subfigure}{\textwidth}
            \centering
            \includegraphics[scale=0.8]{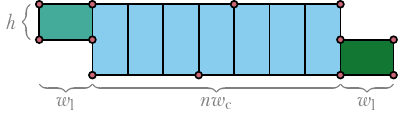}
            \caption{Vertical cylinders in model C.}
            \label{fig:modelC_cylinders}
        \end{subfigure}
        \begin{subfigure}{\textwidth}
            \centering
            \includegraphics[scale=0.8]{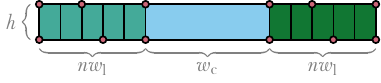}
            \caption{Vertical cylinders in model D.}
            \label{fig:modelD_cylinders}
        \end{subfigure}
        \caption{One-cylinder and two-cylinder models.}
        \label{fig:modelCD}
    \end{figure}
    
    \begin{thm}
    \label{thm:1-2-cyl}
        Let $X$ be a Weierstrass Prym eigenforms in genus three of discriminant $D > 8$. If $X$ admits a one-cylinder or two-cylinder decomposition, then the Teichmüller curve induced by $X$ is arithmetic. Assume now that this is the case, so $D = d^2$ for some integer $d > 4$. We have that:
        \begin{itemize}
            \item if $d$ is even, then $X$ admits both one-cylinder and two-cylinder decompositions;
            \item if $d$ is odd, we have two cases:
            \begin{itemize}
                \item if the HLK-invariant of $X$ is $(3,[0,0,0])$, then $X$ admits a two-cylinder decomposition and no one-cylinder decompositions; and
                \item if the HLK-invariant of $X$ is $(0,[1,1,1])$, then $X$ admits a one-cylinder decomposition and no two-cylinder decompositions.
            \end{itemize}
        \end{itemize}
    \end{thm}
    
    \begin{proof}
        We start by proving that, if that $X$ admits a two-cylinder decomposition, then the Teichmüller curve induced by $X$ is arithmetic. By the classification of cylinder decompositions by Lanneau--Nguyen \cite[Proposition 3.2]{Lanneau-Nguyen:H4}, we can assume that $X$ is a model-C surface, as shown in \Cref{fig:modelC}, with at most three vertical cylinders. Observe that the surface has exactly three vertical cylinders: two lateral cylinders of equal moduli, and one central cylinder. Indeed, the two lateral ``prongs'' of the surface already induce two distinct cylinders, so the rest of the surface must consist of a third central cylinder (as the maximal number of cylinders is three).
        
        Now, let $w_{\mathrm{l}} > 0$ be the width of lateral cylinders, let $w_{\mathrm{c}} > 0$ be the width of the central cylinder, and let $h > 0$ be the height of the lateral cylinders, as in \Cref{fig:modelC_cylinders}. We have that the moduli of the lateral cylinders are $m_{\mathrm{l}} = \frac{w_{\mathrm{l}}}{h}$, and that the modulus of the central cylinder is $m_{\mathrm{c}} = \frac{w_{\mathrm{c}}}{2nh}$, for some integer $n$. Since $X$ is Veech, these moduli are commensurable, so $\frac{m_{\mathrm{l}}}{m_{\mathrm{c}}} = \frac{2n w_{\mathrm{l}}}{w_{\mathrm{c}}} \in \QQ$. We obtain that $\frac{w_{\mathrm{l}}}{w_{\mathrm{c}}} \in \QQ$, so the Teichmüller curve induced by $X$ is arithmetic.
        
        Assume now that $X$ admits a one-cylinder decomposition. We will show that the Teichmüller curve induced by $X$ is arithmetic. Again, by the classification of cylinder decompositions by Lanneau--Nguyen \cite[Proposition 3.2]{Lanneau-Nguyen:H4}, we can assume that $X$ is a model-D surface, as shown in \Cref{fig:modelD}, with at most three vertical cylinders. We claim that this surface has exactly three vertical cylinders: two lateral cylinders of equal moduli, and one central cylinder. Indeed, the way the sides are identified means that the portions of the surface to the left and right of the central cylinder cannot be joined by vertical trajectories. Each of these portions must consist of a single vertical cylinder, as the maximal number of cylinders is three.
        
        Now, let $w_{\mathrm{l}} > 0$ be the width of lateral cylinders, let $w_{\mathrm{c}} > 0$ be the width of the central cylinder, and let $h > 0$ be the height of the central cylinder, as in \Cref{fig:modelD_cylinders}. We have that the moduli of the lateral cylinders is $m_{\mathrm{l}} = \frac{w_{\mathrm{l}}}{nh}$, where $n$ is an integer, and that the modulus of the central cylinder is $m_{\mathrm{c}} = \frac{w_{\mathrm{c}}}{h}$. Since $X$ is Veech, these moduli are commensurable, so $\frac{m_{\mathrm{l}}}{m_{\mathrm{c}}} = \frac{w_{\mathrm{l}}}{nw_{\mathrm{n}}} \in \QQ$. We obtain that $\frac{w_{\mathrm{l}}}{w_{\mathrm{c}}} \in \QQ$, so the Teichmüller curve induced by $X$ is arithmetic.
        
        \begin{figure}
            \centering
            \begin{subfigure}{\textwidth}
                \centering
                \includegraphics[scale=0.8]{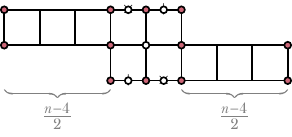}
                \caption{Two-cylinder decomposition with HLK-invariant equal to $(1, [2, 0, 0])$.}
            \end{subfigure}
            \begin{subfigure}{\textwidth}
                \centering
                \includegraphics[scale=0.8]{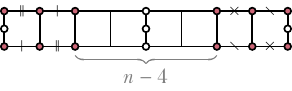}
                \caption{One-cylinder decomposition with HLK-invariant equal to $(1, [2, 0, 0])$.}
            \end{subfigure}
            \caption{Realisability of models C and D for even $d$. The Weierstrass points are marked.}
            \label{fig:modelCD_even}
        \end{figure}
        
        \begin{figure}
            \centering
            \begin{subfigure}{\textwidth}
                \centering
                \includegraphics[scale=0.8]{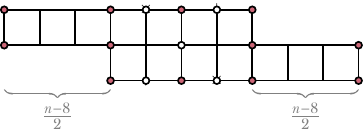}
                \caption{Two-cylinder decomposition with HLK-invariant equal to $(3, [0, 0, 0])$.}
                \label{fig:modelC_odd}
            \end{subfigure}
            \begin{subfigure}{\textwidth}
                \centering
                \includegraphics[scale=0.8]{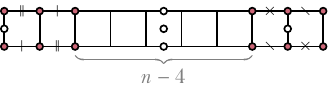}
                \caption{One-cylinder decomposition with HLK-invariant equal to $(0, [1, 1, 1])$.}
                \label{fig:modelD_odd}
            \end{subfigure}
            \caption{Realisability of models C and D for odd $d$. The Weierstrass points are marked.}
        \end{figure}
        
        Assume now that $D = d^2$ for some integer $d > 4$. We will assume then that $X$ is a primitive square-tiled surface, tiled by $n$ unit squares. If $d$ is even, then it is possible to find such surfaces realising both models C and D, as shown in \Cref{fig:modelCD_even}.
        
        We now continue with the case where $d$ is odd and where the HLK-invariant of $X$ is $(3, [0,0,0])$; we have that $n = 2d$. Assume by contradiction that $X$ realises model D. Since $X$ is primitive, the width of the single horizontal cylinder is $1$. This implies that there exist non-integral Weierstrass points, which contradicts the fact that the HLK-invariant of $X$ is $(3, [0,0,0])$. Thus, $X$ does not admit any one-cylinder decomposition. On the other hand, it is possible to find a surface realising model C, as shown in \Cref{fig:modelC_odd}, and thus, $X$ do admit two-cylinder decompositions.
        
        Finally, assume that $d$ is odd and that the HLK-invariant of $X$ is $(0, [1,1,1])$; then $n = d$, so $n$ is odd. Since the two cylinders in model C must be exchanged by the Prym involution and $X$ is primitive, we deduce that this is only possible if $n$ is even. Thus, $X$ does not admit two-cylinder decompositions. On the other hand, it is possible to find a surface realising model D, as shown in \Cref{fig:modelD_odd}, and thus, $X$ do admit one-cylinder decompositions.
    \end{proof}
        
    \begin{rem}
        In the previous proof, we showed an \emph{ad hoc} way to establish that a Weierstrass Prym eigenforms in genus three admitting a one-cylinder decomposition induces an arithmetic Teichmüller curve. Nevertheless, this is a general fact about Veech surfaces.
        Indeed, assume that $X$ is a Veech surface admitting a one-cylinder decomposition.
        Without loss of generality, we can assume that the one-cylinder direction is the horizontal direction, and that the vertical direction is completely periodic. The vertical direction then induces a cylinder decomposition $\{C_i\}_{i \in I}$. For each $i \in I$, let $w_i > 0$ be the width of the cylinder $C_i$, and let $h_i > 0$ be its height. Moreover, let $h > 0$ be the height of the horizontal cylinder. We have that $h_i = n_i h$ for some integer $n_i$, since each vertical cylinder must wind around the height of the vertical cylinder an integral amount of times ($n_i$ is the intersection number of the core curves of the horizontal cylinder and $C_i$).
        Thus, the modulus of $C_i$ is $m_i = \frac{w_i}{h_i} = \frac{w_i}{n_i h}$. Finally, since $X$ is Veech, we have that $\frac{m_i}{m_j} = \frac{w_i}{w_j}\frac{n_j}{n_i} \in \QQ$ for each $i, j \in I$. Therefore, $\frac{w_i}{w_j} \in \QQ$, which shows that the Teichmüller curve induced by $X$ is arithmetic.
        
        At the heart of the previous argument lies the fact that the heights of the cylinders $C_i$ are already commensurable, simply since they all wind around the same horizontal cylinder. Thus, the fact that $X$ is Veech allows us to conclude that their widths are also commensurable.
        Hence, this argument fails when one tries to consider Veech surfaces with more than one horizontal cylinder.
    \end{rem}
    
    \section{Failure of the Kontsevich--Zorich representation to be faithful}
    \label{app:representation}
    
    In this appendix, we show that the representation $\SL(X) \to \Aut(H_1^{(0)})$ is never faithful for a Weierstrass Prym eigenforms in genus three. More precisely, if $X$ is such a surface, recall that the Veech group $\SL(X)$ can be canonically identified with the affine group $\Aff(X)$, since $\Aut(X)$ is trivial. Consider the holonomy map $\mathrm{hol} \colon H_1(X; \ZZ) \to \CC$ and define the \emph{holonomy-free submodule} $H_1^{(0)} \subseteq H_1(X; \ZZ)$ to be the kernel of $\mathrm{hol}$.
    \begin{rem} We caution the reader that the notation $H_1^{(0)}$ is not standard, as it often stands for the kernel of the holonomy map $\mathrm{hol}_\RR \colon H_1(X,\RR) \to \CC \cong \RR^2$ when regarded as a map of $\RR$-vector spaces. With this definition, $H_1^{(0)}$ always has real codimension two. Equivalently, this version of $H_1^{(0)}$ can be described as the symplectic-orthogonal of the tautological plane $H_1^{\mathrm{st}}$, that is, the Poincaré-dual of $\langle \Re(\omega), \Im(\omega)\rangle$. In our context, on the other hand, $H_1^{(0)}$ is a $\ZZ$-module whose rank depends on whether the surface is arithmetic or not (and, more generally, its rank depends on the degree and rank of the corresponding linear invariant suborbifold).
    \end{rem}
    The submodule $H_1^{(0)}$ is invariant for the action of $\SL(X)$. Thus, the Kontsevich--Zorich representation $\varrho \colon \SL(X) \to \Aut(H_1(X; \RR))$ induces a natural subrepresentation $\varrho^{(0)} \colon \SL(X) \to \Aut(H_1^{(0)})$, which is, in general, not known to be faithful. We will show that this representation is never faithful for a Weierstrass Prym eigenforms in genus three.
    
    It was shown by Hooper--Weiss \cite[Theorem 5.6]{Hooper-Weiss:generalized_staircases} (by using some results in an unpublished manuscript of Thurston \cite{Thurston:minimal_stretch}) that the kernel of the action of the Veech group on any invariant submodule of rank two is a Fuchsian group of the first kind (and, in particular, it is not trivial). Moreover, this result can be extended to submodules of higher rank, as shown by the second named author in an unpublished note \cite[Theorem 2]{Pardo:remark} (see also \cite{Crovisier-Hubert-Lanneau-Pardo}), provided that the submodule can be decomposed into invariant blocks of rank two. In other words, the combination of these results shows that:
    \begin{thm}[{\cite{Thurston:minimal_stretch}, \cite[Theorem 5.6]{Hooper-Weiss:generalized_staircases}, \cite[Theorem 2]{Pardo:remark}}]
        \label{thm:two_dimensional}
        Let $X$ be a Veech surface and let $E$ be a submodule of $H_1^{(0)}$ that can be split as $E = E_1 \oplus E_2 \oplus \dotsb \oplus E_n$, where, for each $1 \leq k \leq n$, $E_k$ is an invariant submodule for the action of $\SL(X)$ and has rank two. Then, the kernel of the subepresentation $\SL(X) \to \Aut(E)$ of the Kontsevich--Zorich representation $\varrho \colon \SL(X) \to \Aut(H_1(X; \ZZ))$ is a Fuchsian group of the first kind and, in particular, it is not trivial.
        \qed
    \end{thm}
    
    Using the previous theorem, we can show the main result of this appendix:
    \begin{thm}
        Let $X$ be a Weierstrass Prym eigenforms in genus three. Let
        \[
            \varrho \colon \SL(X) \to \Aut(H_1(X; \ZZ))
        \]
        be the Kontsevich--Zorich representation. Let now $\varrho^{(0)} \colon \SL(X) \to \Aut(H_1^{(0)})$ be the subrepresentation of $\varrho$ induced on $H_1^{(0)}$. Then, $\varrho^{(0)}$ is not faithful.
    \end{thm}
    
    \begin{proof}
        By \Cref{thm:two_dimensional}, it is enough to show that $H_1^{(0)}$ can be decomposed into a direct sum of invariant submodules of rank two. We will do this separately for the arithmetic and non-arithmetic cases.
        
        Consider first the case where the Teichmüller curve induced by $X$ is arithmetic. We have the splitting
        \[
            H_1(X; \RR) = H_1^+ \oplus H_1^-,
        \]
        where $H_1^+$ is invariant submodule for the action of the Prym involution $\iota$, and $H_1^-$ is its anti-invariant submodule. Since $\iota^*(\omega) = -\omega$, the anti-invariant submodule $H_1^-$ has rank four and contains the restriction $H_1^{\mathrm{st}}(\ZZ)$ of the tautological plane to integral homology (namely, $H_1^{\mathrm{st}}(\ZZ) = H_1^{\mathrm{st}} \cap H_1(X; \ZZ)$), which, in turn, has rank two; while the invariant submodule $H_1^+$ has rank two and is contained in $H_1^{(0)}$, which is, in turn, has rank four. In other words, if we set $F = H_1^- \cap H_1^{(0)}$, we have that $F$ is a rank-two invariant submodule satisfying
        \[
            H_1(X; \ZZ) = H_1^{\mathrm{st}}(\ZZ) \oplus F \oplus H_1^+, \quad H_1^- = H_1^{\mathrm{st}}(\ZZ) \oplus F, \quad H_1^{(0)} = F \oplus H_1^+,
        \]
        so we conclude that $\varrho^{(0)}$ is not faithful by \Cref{thm:two_dimensional}.

        On the other hand, if $X$ is non-arithmetic, the holonomy-free submodule $H_1^{(0)}$ has rank two. Indeed, in this case we have that $H_1^-$ is the rank-four submodule induced by the tautological plane $H_1^{\mathrm{st}}$ and its Galois-conjugate $(H_1^{\mathrm{st}})^\sigma$ (namely, $H_1^- = (H_1^{\mathrm{st}} \oplus (H_1^{\mathrm{st}})^\sigma) \cap H_1(X; \ZZ)$) and, therefore, $H_1^{(0)}$ coincides with the rank two subbmodule $H_1^+$.
        Thus, the fact that $\varrho^{(0)}$ is not faithful follows immediately from \Cref{thm:two_dimensional}.
    \end{proof}

\sloppy
\printbibliography

\end{document}